\documentclass[reqno]{amsart}
\usepackage[utf8]{inputenc}
\usepackage[T1]{fontenc}
\usepackage{lmodern}
\usepackage{hyperref}
\usepackage{amsmath}
\usepackage{amsfonts}
\usepackage{amssymb}
\usepackage{cancel}
\usepackage{color}
\usepackage{lipsum}

\numberwithin{equation}{section}
\newtheorem{theorem}{Theorem}[section]
\newtheorem{proposition}[theorem]{Proposition}
\newtheorem{remark}[theorem]{Remark}
\newtheorem{definition}{Definition}[section]
\newtheorem{claim}{Claim}[section]
\newtheorem{assumption}{Assumption}
\newtheorem{lemma}{Lemma}[section]

\newcommand{\<}{\langle}
\renewcommand{\>}{\rangle}
\renewcommand{\H}{\mathbb{H}}
\begin{document} 
\title{ Generalized spectrum of Steklov-Robin problem for elliptic system}

%\author[M. N. Nkashama\hfil]
%{M. N. Nkashama}

%\address{M. N. Nkashama \hfill\break
%Department of Mathematics, University of Alabama at
%Birmingham \hfill\break
%Birmingham, Alabama 35294-1170, USA }
%\email{ nkashama@@math.uab.edu }

\author[Alzaki Fadlallah\hfil ]
{Alzaki.M.M. Fadlallah}

\address{Alzaki.M.M. Fadlallah \hfill\break
Department of Mathematics, University of Alabama at
Birmingham \hfill\break
Birmingham, Alabama 35294-1170, USA} 
\email{ zakima99@@math.uab.edu }

%\date

\begin{abstract}
We will study the generalized Steklov–Robin eigensystem (with possibly matrices weights) in which the spectral
parameter is both in the system and on the boundary. We prove the
existence of of an increasing unbounded sequence of eigenvalues 
. The method of proof makes use of variational arguments.
\end{abstract}
\maketitle

\section{Introduction}
\begin{equation}\label{Es2}
\begin{gathered}
 -\Delta U+ A(x)U =\mu M(x)U \quad\text{in } \Omega,\\
\frac{\partial U}{\partial \nu}+\Sigma(x)U=\mu P(x)U \quad\text{on }
\partial\Omega,
\end{gathered}
\end{equation}
where $\Omega\subset\mathbb{R}^N$ , $N\geq 2$ is a bounded domain  with boundary
$\partial \Omega$ of class $C^{0,1}$,

$U=[u_1,\dots,u_k]^{T}\in H(\Omega)$
and the matrix 
$$A(x)=\left[
\begin{array}{cccc}
 a_{11}(x)&a_{12}(x)&\cdots&a_{1k}(x)\\ 
a_{21}(x)&a_{22}(x)&\cdots&a_{2k}(x)\\ 
\vdots & \vdots  &\ddots &\vdots  \\ 
a_{k1}(x)&a_{k2}(x)&\cdots&a_{kk}(x)
\end{array}\right]
$$\\
Verifies the following conditions:
\begin{enumerate}

\item[$(A1$)] The functions $a_{ij}:\Omega\to\mathbb{R},$ $a_{ij}(x)\ge 0,~~ \forall~ i,j=1,\cdots,k, ~~x\in\Omega$
with strict inequality on a set of positive measure of $\Omega$.
 \item[$(A2)$] $A(x)$ is positive semidefinite matrix on $\mathbb{R}^{{k}\times{k}},$ almost everywhere $x\in\Omega,$ and 
 $A(x)$ is positive definite on a set of positive measure of $\Omega$ with $a_{ij}\in L^{p}(\Omega)~~ \forall ~i,j=1,\cdots,k,$ 
for $p>\frac{N}{2}$ when $N\geq 3$, and $p>1$ when $N=2$
% \item A(x) is symmetric i.e.; $a_{ij}(x)=a_{ji}(x)$  forall $x\in\Omega$
 
 \end{enumerate}
 
$\partial/\partial\nu :=\nu\cdot\nabla$
is the outward (unit) normal derivative on $\partial\Omega$,
The matrix $\Sigma$ is 
$$\Sigma(x)=\left[
\begin{array}{cccc}
 \sigma_{11}(x)&\sigma_{12}(x)&\cdots&\sigma_{1k}(x)\\ 
\sigma_{21}(x)&\sigma_{22}(x)&\cdots&\sigma_{2k}(x)\\ 
\vdots & \vdots  &\ddots &\vdots  \\ 
\sigma_{k1}(x)&\sigma_{k2}(x)&\cdots&\sigma_{kk}(x)
\end{array}\right]
$$
Verifies the following conditions:\\

\begin{enumerate}

\item[$(S1)$] The functions $\sigma_{ij}:\partial\Omega\to\mathbb{R},$ $\sigma_{ij}(x)\ge 0,~~ \forall ~i,j=1,\cdots,k ~~x\in\partial\Omega$
\item[$(S2)$] $\Sigma(x)$ is positive semidefinite matrix on $\mathbb{R}^{{k}\times{k}},$ almost everywhere $x\in\partial\Omega,$
and  $\Sigma(x)$ is positive definite on a set of positive measure of $\partial\Omega$
  with $\sigma_{ij}\in L^{q}(\partial\Omega)~~ \forall~ i,j=1,\cdots,k,$ 
for $q\geq N-1$ when $N\geq 3$, and $q>1$ when $N=2$
% \item A(x) is symmetric i.e.; $a_{ij}(x)=a_{ji}(x)$  forall $x\in\Omega$
\end{enumerate}

\begin{remark}
 \item Conditions $(A2),(S2)$ are equivalent to 
 $$\int_{\Omega}\<A(x)U,U\>\,dx+\int_{\partial\Omega}\<\Sigma(x)U,U\>\,dx > 0$$
\item We can put condition $(S2)$ in $A$ and condition $(A2)$ in $\Sigma$ i.e.;  interchange conditions in $A$ and $\Sigma$ 
 \end{remark}
Note that the eigensystem (\ref{Es2}) includes as special cases the weighted Steklov eigenproblem for a class of elliptic system when ( 
$\Sigma(x)=M(x)\equiv 0$, $A(x)\in\mathbb{R}^{2\times{2}}$ and  $P(x)=I$  where $I$ 
is identity matrix that was considered in \cite{GMR13}. For scalar  case 
generalized Steklov-Robin  spectrum of the scalar case that was considered in \cite{Mav12}
as well as weighted Robin-Neumann eigenproblem when ($P(x)\equiv 0$ and $M(x)=I$) that was considered in \cite{Auc12} and \cite{GMR13}. 
Under conditions of $A(x),\Sigma(x)$ together with the hypothesis on $\Omega$, we have 
\section{Notations definitions}
To put our results into the context, we have collected in this shore section some relevant notations and definitions for our purposes. 
Throughout this work, $H^{1}_{0},$ $H^{1}(\Omega)$ denotes the usual real Sobolev space of functions on $\Omega$. 
\begin{itemize}
 \item $k\in\mathbb{N}$
 \item $H(\Omega)=H^{1}(\Omega)\times H^{1}(\Omega)\times\cdots\times H^{1}(\Omega)\times H^{1}(\Omega)=[H^{1}(\Omega)]^{k}$
 \item $H_{0}(\Omega)=H^{1}_{0}(\Omega)\times H^{1}_{0}(\Omega)\times\cdots\times H^{1}_{0}(\Omega)\times H^{1}_{0}(\Omega)=[H^{1}_{0}(\Omega)]^{k}$
 \item $L^{p}_{k}(\Omega)=L^{p}(\Omega)\times L^{p}(\Omega)\times\cdots\times L^{p}(\Omega)\times L^{p}(\Omega)=[L^{p}(\Omega)]^{N}$
 \item  $L^{p}_{k}(\partial\Omega)=L^{p}(\partial\Omega)\times L^{p}(\partial\Omega)\times\cdots\times L^{p}(\partial\Omega)\times L^{p}(\partial\Omega)=[L^{p}(\partial\Omega)]^{k}$
\end{itemize}
\begin{definition}{\bf{Cooperative-plus}} 
Cooperative-plus matrix means that all the entries of the matrix are non-negative
\end{definition}
\begin{itemize}

 \item The embedding of $H^{1}(\Omega)$ into $L^{p}(\Omega)$ is continuous for $1\leq p\leq p(N)$ and compact for 
 $1\leq p\leq p(N)$ where $p(N)=\frac{2N}{N-2}$ if $N\geq 3$ and $p(N)=\infty$ if $N=2$ (see \cite{EG} for more details)
   \item $\<x,y\>=\sum_{i=1}^{k}x_iy_i~~\forall~ x,y\in H(\Omega)$ then 
   \begin{equation}\label{NAS}
\begin{gathered}
\<U,V\>_{(A,\Sigma)}=\int_{\Omega}\left[\triangledown U.\triangledown V+\<A(x)U,V\>\right]~~dx+\int_{\partial\Omega}\<\Sigma(x) U,V\>\,dx
\end{gathered}
\end{equation}
  
  defines an inner product for $H(\Omega)$, with associated norm $||.||_{(A,\Sigma)}$ 
  \begin{equation}\label{IAS}
\begin{gathered}
||U||_{(A,\Sigma)}=\int_{\Omega}\left[\triangledown U.\triangledown U+\<A(x)U,U\>\right]~~dx+\int_{\partial\Omega}\<\Sigma(x) U,U\>\,dx
\end{gathered}
\end{equation}
  
  this norm is equivalent to the standard $H(\Omega)-$norm
  \end{itemize}
  Claim: $<.,.>_{(A,\Sigma)}$  is inner product on $H(\Omega)$
  
  \begin{proof}
  Let $U, V$ and $W$  be vectors on $\in H(\Omega)$ and $\alpha$ be a scalar in $\mathbb{R}$, then:
  \begin{enumerate}
   \color{blue}{
   \item $$\<U+V,W\>_{(A,\Sigma)}=$$ $$\int_{\Omega}\left[\triangledown (U+V).\triangledown W+\<A(x)(U+V),W\>\right]~~dx+
   \int_{\partial\Omega}\<\Sigma(x) U+V,W\>\,dx$$
   $$=\int_{\Omega}\left[\triangledown U.\triangledown W+\<A(x)U,W\>\right]~~dx+\int_{\partial\Omega}\<\Sigma(x) U,W\>\,dx =$$$$
   \int_{\Omega}\left[\triangledown V.\triangledown W+\<A(x)V,W\>\right]~~dx+\int_{\partial\Omega}\<\Sigma(x) V,W\>\,dx=\<U,W\>_{(A,\Sigma)}+\<V,W\>_{(A,\Sigma)}$$
   \item $$\<\alpha U,W\>_{(A,\Sigma)}=\int_{\Omega}\left[\triangledown \alpha U.\triangledown W+\<A(x)\alpha U,W\>\right]\,dx+
   \int_{\partial\Omega}\<\alpha\Sigma(x) U,W\>\,dx=$$$$
   \alpha\left(\int_{\Omega}\left[\triangledown U.\triangledown W+\<A(x)U,W\>\right]+\int_{\partial\Omega}\<\sigma(x) U,W\>\,dx\right)
   =\alpha\<U,W\>_{(A,\Sigma)}$$
   \item Since $A,\Sigma$ are symmetric and the $\<x,y\>=\<y,x\>$ 
   $$\<U,V\>_{(A,\Sigma)}=\int_{\Omega}\left[\triangledown U.\triangledown V+\<A(x)U,V\>\right]+\int_{\partial\Omega}\<\Sigma(x) U,V\>\,dx
   $$$$=\int_{\Omega}\left[\triangledown U.\triangledown V+\<U,A(x)V\>\right]+\int_{\partial\Omega}\< U,\Sigma(x) V\>\,dx=$$$$
   \int_{\Omega}\left[\triangledown V.\triangledown U+\<A(x)V,U\>\right]+\int_{\partial\Omega}\< \Sigma(x) V,U\>\,dx=\<V,U\>_{(A,\Sigma)}$$ 
   \item Since we have that $\<A(x)U,U\>\geq 0~\forall~U\in H(\Omega)$ on $\Omega$ (semidefinite)
   and $\<\Sigma(x)U,U\>\geq 0~\forall~U\in H(\Omega)$ on $\partial\Omega$ (semidefinite)
   so we have that 
   $$\<U,U\>_{(A,\Sigma)}=\int_{\Omega}\left[\triangledown U.\triangledown U+\<A(x)U,U\>\right]=
   \int_{\Omega}\left[|\triangledown U|^{2}+\<A(x)U,U\>\right]+\int_{\partial\Omega}\< \Sigma(x) U,U\>\,dx\geq 0 $$
   if $$\int_{\Omega}\left[|\triangledown U|^{2}+\<A(x)U,U\>\right]+\int_{\partial\Omega}\< \Sigma(x) U,U\>\,dx=0$$
   then $$\<A(x)U,U\>=0$$ iff $U=0$ on $\Omega$,
   and $$\<\Sigma(x)U,U\>=0$$ iff $U=0$ on $\partial\Omega$. \\
   So, $\<U,U\>_{(A,\Sigma)}\geq 0,$  $\<U,U\>_{(A,\Sigma)}=0$ iff $U=0$ 
   }
    \end{enumerate}
   Indeed $\<.,.\>_{(A,\Sigma)}$ is inner product on $H(\Omega)$
   \end{proof}

\section{Generalized Steklov-Robin eigensystem}
In this section, we will study the generalized spectrum that will  be used for the comparison with the nonlinearities in the system (\ref{Es1}).
This spectrum includes the Steklov, Neumann and Robin spectra, We therefore generalize the results in \cite{Auc12}, and \cite{GMR13}\\.
Consider the elliptic system 
\begin{equation}\label{Es3}
\begin{gathered}
 -\Delta U+ A(x)U =\mu M(x)U \quad\text{in } \Omega,\\
\frac{\partial U}{\partial \nu}+\Sigma(x)U=\mu P(x)U \quad\text{on }
\partial\Omega,
\end{gathered}
\end{equation}
Where 
$$M(x)=\left[
\begin{array}{cccc}
 m_{11}(x)&m_{12}(x)&\cdots&m_{1k}(x)\\ 
m_{21}(x)&m_{22}(x)&\cdots&m_{2k}(x)\\ 
\vdots & \vdots  &\ddots &\vdots  \\ 
m_{k1}(x)&m_{k2}(x)&\cdots&m_{kk}(x)
\end{array}\right]
$$
and 
$$P(x)=\left[
\begin{array}{cccc}
 \rho_{11}(x)&\rho_{12}(x)&\cdots&\rho_{1k}(x)\\ 
\rho_{21}(x)&\rho_{22}(x)&\cdots&\rho_{2k}(x)\\ 
\vdots & \vdots  &\ddots &\vdots  \\ 
\rho_{k1}(x)&\rho_{k2}(x)&\cdots&\rho_{kk}(x)
\end{array}\right]
$$
\begin{enumerate}
 \item[(M1)] Where $M(x)$ is positive definite matrix on $\mathbb{R}^{{k}\times{k}}$, almost everywhere $x\in\Omega,$ 
The functions $m_{ij}:\Omega\to\mathbb{R},$ $m_{ij}(x)\ge 0,~~ \forall~ i,j=1,\cdots,k ~~x\in\Omega,$
with strict inequality on a set of positive measure of $\Omega$, $m_{ij}\in L^{p}(\Omega)~~ \forall~ i,j=1,\cdots,k,$ 
 for $p\geq \frac{N}{2}$ when $N\geq 3$, and $p>1$ when $N=2.$ 
\item[(P1)] Where $P(x)$ is positive semidefinite matrix on $\mathbb{R}^{{k}\times{k}},$ almost everywhere $x\in\partial\Omega,$ and positive define
 on a set of positive measure of $\partial\Omega$ with $\rho_{ij}\in L^{q}(\partial\Omega)~~ \forall i,j=1,\cdots,k,$ 
for $q\geq N-1$ when $N\geq 3$, and $q>1$ when $N=2.$
\item[(MP)] And $(m_{ij},\rho_{ij})>0$; that is  $m_{ij}(x)\geq 0$ a.e in $\Omega$, $\rho_{ij}(x)\geq 0$ a.e on $\partial\Omega$ for all $i,j=1,\cdots,k$
such that $$\int_{\Omega}m_{ij}(x)dx+\int_{\partial\Omega}\rho_{ij}(x)dx >0 ~\forall~i,j=1,\cdots,k,$$
 
\end{enumerate}
Assume that $A(x),~\Sigma(x),~M(x),~P(x)$ are Verifies the following assumption :\\
\begin{assumption}\label{ASMP}
 \begin{enumerate}
$A(x)$ is positive definite on a set of positive measure of $\Omega$ with $A(x)\in L^{p}(\Omega)$ 
for $p>\frac{N}{2}$ when $N\geq 3$, and $p>1$ when $N=2.$ 
\item [OR] 
  $\Sigma(x)$ is positive definite on a set of positive measure of $\partial\Omega$
  with $\Sigma(x)\in L^{q}(\partial\Omega)$
for $q\geq N-1$ when $N\geq 3$, and $q>1$ when $N=2$
\item[AND] 
$M(x)$ is positive definite on a set of positive measure of $\Omega$ with $M(x)\in L^{p}(\Omega)$ 
for $p>\frac{N}{2}$ when $N\geq 3$, and $p>1$ when $N=2.$ 
\item[OR]
$P(x)$ is positive definite on a set of positive measure of $\partial\Omega$
  with $P(x)\in L^{q}(\partial\Omega)$
for $q\geq N-1$ when $N\geq 3$, and $q>1$ when $N=2.$
\end{enumerate}
\end{assumption}
\begin{remark}\label{rm1}
 Since $A(x),~\Sigma(x),~M(x),~P(x)$ are satisfied $(A2),~ (S2),~ (M1), ~(P1)$ respectively and they are Cooperative-plus matrices,
 then we can write them in following from 
 (i.e.;  eigen-decomposition of a positive semi-definite matrix or diagonalization)
 $$A(x)=Q_{A}^{T}(x)D_{A}(x)Q_{A}(x).$$
$$\Sigma(x)=Q_{\Sigma}^{T}(x)D_{\Sigma}(x)Q_{\Sigma}(x).$$
$$M(x)=Q_{M}^{T}(x)D_{M}(x)Q_{M}(x).$$
$$P(x)=Q_{P}^{T}(x)D_{P}(x)Q_{P}(x).$$
Where $Q_{J}(x)^{T}Q_{J}(x)=I$ ($Q_{J}^{T}(x)=Q_{J}^{-1}(x)$ i.e.; are orthogonal matrix ) are the normalized eigenvectors,  $I$ is identity matrix, $D_{J}(x)$ is diagonal
matrix in the diagonal of $D_{J}(x)$ the eigenvalues of $J(x)$ (i.e.; $D(x)_{J}=diag\left(\lambda_{1}^{J}(x),\cdots,\lambda_{k}^{J}(x)\right)$)
and $J(x)=\{A(x),~\Sigma(x),~M(x),~P(x)\}$
 \end{remark}

\begin{remark}
 The weight matrices $M(x)$ and $P(x)$ may vanish on subset of positive measure. 
\end{remark}
\begin{definition}
 The generalized Steklov-Robin eigensystem is to find a pair $(\mu,\varphi)\in\mathbb{R}\times H(\Omega)$ with 
 $\varphi\not\equiv 0$ such that 
 \begin{equation}\label{Es4}
\begin{gathered}
 \int_{\Omega}\triangledown\varphi.\triangledown U\,dx+\int_{\Omega}\<A(x)\varphi,U\>\,dx+\int_{\partial\Omega}\<\Sigma(x)\varphi,U\>\,dx= \\
 \mu\left[\int_{\Omega}\<M(x)\varphi,U\>\,dx+\int_{\partial\Omega}\<P(x)\varphi,U\>\,dx\right] ~~\forall~~U\in H(\Omega)
\end{gathered}
\end{equation}
\end{definition}
\begin{remark}
 Let $U=\varphi$ in (\ref{Es4}), if there is such an eigenpair, then 
$\mu>0$ and 
$$\int_{\Omega}\<M(x)\varphi,\varphi\>\,dx+\int_{\partial\Omega}\<P(x)\varphi,\varphi\>\,dx >0$$
$$\left(i.e; \sum_{j=1}^{k}\sum_{i=1}^{k}\left(\int_{\Omega}m_{ij}(x)\varphi_{j}\varphi_{i}\,dx+\int_{\partial\Omega}\rho_{ij}(x)\varphi_{j}\varphi_{i}\,dx\right)>0\right)$$ and 
\end{remark}
\begin{remark}
If $\int_{\Omega}\<M(x)\varphi,\varphi\>\,dx+\int_{\partial\Omega}\<P(x)\varphi,\varphi\>\,dx =0.$
Then 
$$\int_{\Omega}|\triangledown\varphi|^{2}\,dx+\int_{\Omega}\<A(x)\varphi,\varphi\>\,dx+\int_{\partial\Omega}\<\Sigma(x)\varphi,\varphi\>\,dx=0$$
We have that 
$\int_{\Omega}|\triangledown\varphi|^{2}\,dx=0$ this implies that $\varphi={\rm~ constant}$
and 

$\int_{\Omega}\<A(x)\varphi,\varphi\>\,dx=0$ this implies that $A(x)=0,~a.e.$  in $\Omega$ and 
$\int_{\partial\Omega}\<\Sigma(x)\varphi,\varphi\>\,dx =0,$ then  $\Sigma(x)=0~a.e$ on $\partial\Omega.$
So we have that, $\varphi$ would be a constant vector function; which would contradict the assumptions (assumption \ref{ASMP})  imposed on 
$A(x)$ and $\Sigma(x)$
\end{remark}
\begin{remark}
If $A(x)\equiv 0$ and $\Sigma\equiv 0$ then $\mu=0$ is an eigenvalue of the system (\ref{Es3}) with eigenfunction 
$\varphi={\rm constant~ vector ~function}$ on $\bar{\Omega},$
\end{remark}
It is therefore appropriate to consider the closed linear subspace of $H(\Omega)$ under assumption \ref{ASMP} defined by 
$$\H_{(M,P)}(\Omega):=\{U\in H(\Omega):\int_{\Omega}\<M(x)U,U\>\,dx+\int_{\partial\Omega}\<P(x)U,U\>\,dx=0\}.$$
Now all the  eigenfunctions associated with (\ref{Es4}) belongs  to the $(A,\Sigma)-$orthogonal complement $H_{(M,P)}(\Omega):=[\H_{(M,P)}(\Omega)]^{\perp}$ 
$$i.e.;H_{(M,P)}(\Omega)=[\H_{(M,P)}(\Omega)]^{\perp}=$$
$$\{U\in H(\Omega):\int_{\Omega}\<M(x)U,U\>\,dx+\int_{\partial\Omega}\<P(x)U,U\>\,dx>0~{\rm and} ~~\<U,V\>_{(A,\Sigma)}=0, \forall~V\in \H_{(M,P)}(\Omega)\}$$
of this subspace in $H(\Omega)$
\begin{definition}
 $\Omega(M):=\{x\in\Omega:M(x) >0\}.$ \\ 
  $\partial\Omega(P):=\{x\in\partial\Omega:P(x) >0\}.$ 
\end{definition}
We will show that indeed the  $H_{(M,P)}(\Omega)$ is subspace of $H(\Omega).$
Let $U,~ V\in H_{(M.P)}(\Omega)$ and $\alpha\in\mathbb{R}$ we show that $\alpha U\in H_{(M.P)}(\Omega)$ and $U+V\in H_{(M.P)}(\Omega) $
$$\left(\int_{\Omega}\< M(x)(\alpha U),\alpha U\>\,dx+\int_{\partial\Omega}\<P(x)(\alpha U),\alpha  U\>\,dx\right)=$$
$$\alpha^{2}\left(\int_{\Omega}\<M(x)U,U\>\,dx+\int_{\partial\Omega}\<P(x)U,U\>\,dx\right)=^{U\in H_{(M.P)}}0$$
Therefore $\alpha U \in H_{(M.P)}(\Omega)$. Now we show that $U+V\in H_{(M.P)}(\Omega) $
$$\int_{\Omega}\< M(x)(U+V),(U+V)\>\,dx+\int_{\partial\Omega}\<P(x)(U+V),(U+V)\>\,dx=$$
$$\int_{\Omega}\<M(x)U,U\>\,dx+\int_{\partial\Omega}\<P(x)U,U\>\,dx+\int_{\Omega}\<M(x)V,V\>\,dx+\int_{\partial\Omega}\<P(x)V,V\>\,dx$$
$$+2\int_{\Omega}\< M(x)U,V\>\,dx+2\int_{\partial\Omega}\<P(x)U,V)\>\,dx$$
{\bf{Case 1}}\\
If $x\in  \partial\Omega(P),$  this implies that   $U=0$ on $\partial\Omega(P)$
\begin{itemize}
\item $M(x)=0$  on $\Omega$ (Steklov problem)
\item $M(x)>0$ on set of positive measure of $\Omega$ this implies that $U=0$ on $\Omega$
Now we use  and \ref{ASMP}, \ref{rm1} we have that 
$$0=\int_{\Omega}\<M(x)U,U\>\,dx=\int_{\Omega}\<Q^{T}_{M}(x)D_{M}(x)Q_{M}(x)U,U\>\,dx=
\int_{\Omega}\<D_{M}(x)Q_{M}(x)U,Q_{M}(x)U\>\,dx
$$ 
Since $D_{M}(x)>0$ this implies that $Q_{M}(x)U=0$ since $Q_{M}(x)$ is invertible this implies that 
$U=0$ on $\Omega(M)$
\end{itemize}
{\bf{Case 2}}\\
If $x\in \Omega(M),$  this implies that   $U=0$ on $\Omega$
\begin{itemize}
\item $P(x)=0$  on $\partial\Omega$ (usual spectrum problem ) 
\item $P(x)>0$ on set of positive measure of $\partial\Omega$ this implies that $U=0$ on $\partial\Omega$
Now we use  and \ref{ASMP}, \ref{rm1} we have that 
$$0=\int_{\partial\Omega}\<P(x)U,U\>\,dx=\int_{\partial\Omega}\<Q^{T}_{P}(x)D_{P}(x)Q_{P}(x)U,U\>\,dx=
\int_{\partial\Omega}\<D_{P}(x)Q_{P}(x)U,Q_{P}(x)U\>\,dx
$$ 
Since $D_{P}(x)>0$ this implies that $Q_{P}(x)U=0$ since $Q_{P}(x)$ is invertible this implies that 
$U=0$ on $\Omega(P)$
\end{itemize}
Therefor $U+V\in \H_{(M,P)}(\Omega), $ so we have that $\H_{(M,P)}(\Omega)$ subspace of $H(\Omega).$ 
So that 
$$\H_{(M,P)}(\Omega):=\{U\in H(\Omega):U=0~a.e.;~{\rm in}~\Omega(M)~{\rm and}~\Gamma U=0~a.e.;{\rm on}~~ \partial\Omega(P)\}$$
\begin{remark}.\\
\begin{enumerate}
 \item If $M(x)\equiv 0$ in $\Omega$ and $x\in\partial\Omega(P)$, then the subspace $\H_{(M,P)}(\Omega)=H_{0}(\Omega)$
\item If $P(x)\equiv 0$ in $\partial\Omega$ and $x\in\Omega(M)$, then the subspace $\H_{(M,P)}(\Omega)=\{0\}$
\end{enumerate}
\end{remark}
Thus, one can split the Hilbert space $H(\Omega)$ as a driect $(A,\Sigma)-$orthogonal sum in the following way 
$$H(\Omega)=\H_{(M,P)}(\Omega)\oplus_{(A,\Sigma)}[\H_{(M,P)}(\Omega)]^{\perp}$$
\begin{remark}
 If $(\mu,\varphi)\in\mathbb{R}\times H(\Omega)$ is an eigenpair of \ref{Es4}, then it follows from the definition of $H_{0}(\Omega)$
that 
$$<\varphi,V>_{(A,\Sigma)}=
\int_{\Omega}[\triangledown\varphi.\triangledown V+\<A(x)\varphi,V\>]\,dx
+\int_{\partial\Omega}\<\Sigma(x)\varphi,V\>\,dx=0,~~~$$
$\forall ~V\in \H_{(M,P)}(\Omega)$ that is, $\varphi\in [\H_{(M,P)}(\Omega)]^{\bot} $
\end{remark}
\begin{itemize}
  \item We shall make use in what follows the real Lebesgue space $L^{q}_{k}(\partial\Omega)$ for $1\leq q\leq\infty$, and 
of the continuity and compactness of the trace operator. 
$$\Gamma :H(\Omega)\to L^{q}_{k}(\partial\Omega)~~for~~1\leq q<\frac{2(n-1)}{n-2}$$
 is well-defined it is a Lebesgue integrable function  with respect to Hausdorff $N-1$ dimensional measure, 
 sometime we will just use $U$ in place of $\Gamma U$ when considering the trace of function on $\partial\Omega$
Throughout  this work we denote the $L^{2}_{N}(\partial\Omega)-$ inner product by 
$$\<U,V\>_{\partial}:=\int_{\partial\Omega}U.V\,dx$$ and  the associated norm by 
$$||U||^{2}_{\partial}:=\int_{\partial\Omega}U.U~\forall U,V\in H(\Omega)$$
(see \cite{KJF77}, \cite{Nec67} and the references therein for more details )
  \item The trace mapping $\Gamma:H(\Omega)\to L^{2}_{N}(\partial\Omega)$ is compact (see \cite{Gri})
  \item 
  \begin{equation}\label{IMP}
\begin{gathered}
\<U,V\>_{(M,P)}=\int_{\Omega}\<M(x)U,V\>\,dx+\int_{\partial\Omega}\<P(x) U,V\>\,dx
\end{gathered}
\end{equation}
    defines an inner product for $H(\Omega)$, with associated norm 
  \begin{equation}\label{NMP}
\begin{gathered}
  ||U||_{(M,P)}^{2}:=\int_{\Omega}\<M(x)U,U\>\,dx+\int_{\partial\Omega}\<P(x) U,U\>\,dx
\end{gathered}
\end{equation}
\end{itemize}
\section{Preliminary Results}
  In this section we study some auxiliary results, which will be need in the sequel for the proof of our main results.
   Using the H\"{o}lder inequality, the continuity of the trace operator, the Sobolev embedding theorem and lower semicontinuity of 
 $||.||_{(A,\Sigma)}$, we deduce that $||.||_{(A,\Sigma)}$ (see (\ref{NAS})) is equivalent to standard norm $H(\Omega)-$norm. 
 This observation enables us to prove the existence of an unbounded and discrete spectrum for the Steklov-Robin eigensystem (\ref{Es3}), and 
 discuss some of its properties.\\
 \begin{definition}
 Define the functionals 
 $$\Lambda_{A,\Sigma}:H(\Omega)\to[0,\infty)$$ defined by 
   $$\Lambda_{A,\Sigma}(U):=\int_{\Omega}\left[\triangledown U.\triangledown U+\<A(x)U,U\>\right]~~dx+\int_{\partial\Omega}\<\Sigma(x) U,U\>\,dx=||U||_{(A\Sigma)}^{2},~~\forall~~U\in H(\Omega)$$
   $$\Upsilon_{M,P}:H(\Omega)\to[1,\infty)$$ defined by 
   $$\Upsilon_{M,P}:=\int_{\Omega}\<M(x)U,U\>\,dx+\int_{\partial\Omega}\<P(x)U,U\>\,dx-1=||U||_{(M,P)}^{2}-1,~\forall ~U\in H(\Omega)$$

   \end{definition}

   \begin{lemma}\label{ls}
   Suppose $(A2,~S2,~C1-C3),$ and assumption\ref{ASMP} are holds then the functionals
   $\Lambda_{A,\Sigma}$ and $\Upsilon_{M,P}$ are $C^{1}-$ functional (i.e.; continuous differentiable) 
\begin{proof}
  We compute $\Lambda_{A,\Sigma}'(U)V.$ Let $U,V\in H(\Omega)$  We have that
  $$\Lambda_{A,\Sigma}(U+tV)=\int_{\Omega}|\triangledown (U+tV)|^2\,dx+\int_{\Omega}\<A(x)(U+tV),U+tV\>\,dx
+\int_{\partial\Omega}\<\Sigma(x)(U+tV),U+tV\>\,dx=
$$
$$\int_{\Omega}(|\triangledown U|^2+2t\triangledown U.\triangledown V\,dx+t^2|\triangledown V|^2)\,dx+
\int_{\Omega}\<A(x)U,U\>\,dx+2t\int_{\Omega}\<A(x)U,V\>\,dx+t^{2}\int_{\Omega}\<A(x)V,V\>\,dx +$$
$$\int_{\partial\Omega}\<\Sigma(x)U,U\>+2t<\Sigma U,V\>+t^{2}\<\Sigma(x)V,V\>\,dx$$ 
$$\Lambda_{A,\Sigma}(U)=\int_{\Omega}|\triangledown U|^2\,dx+\int_{\Omega}\<A(x)U,U\>\,dx
+\int_{\partial\Omega}\<\Sigma(x)U,U\>\,dx,
$$
  
{\begin{equation}\label{s5}
\begin{gathered}
\lim_{t\to 0}t^{-1}[(\Lambda_{A,\Sigma}(U+tV)-\Lambda_{A,\Sigma}(U)]=\\
2\left(\int_{\Omega}\triangledown U.\triangledown V \,dx+\int_{\Omega}\<A(x)U,V\> \,dx+\int_{\partial\Omega}\<\Sigma U,V\>\,dx\right).\\
\therefore\Lambda_{A,\Sigma}'(U)V=2\<U,V\>_{(A,\Sigma)}
,~~\forall~~ U,V\in H(\Omega) \\
\end{gathered}
\end{equation}}
Now you compute $\Upsilon_{M,P}'(U)V,$ let $ U,V\in H(\Omega)$
$$\Upsilon_{M,P}(U+tV)=\int_{\Omega}\<M(x)(U+tV),U+tv\>\,dx+\int_{\partial\Omega}\<P(x)(U+tV),U+tV\>\,dx-1=$$
$$t\int_{\Omega}\<M(x)U,U\>\,dx+2t\int_{\Omega}\<M(x)U,V\>\,dx+t^{2}\int_{\Omega}\<M(x)V,V\>\,dx$$ $$
+\int_{\partial\Omega}\<P(x)U,U\>\,dx+2t\int_{\partial\Omega}\<P(x)U,V\>\,dx+t^{2}\int_{\partial\Omega}\<P(x)V,V\>\,dx-1.$$
\begin{equation}\label{6}
\begin{gathered}
\lim_{t\to 0}t^{-1}[\Upsilon_{M,P}(U+tV)-\Upsilon_{M,P}(U)]=2\<U,V\>_{M,P}
\therefore\Upsilon_{M,P}'(U)V=2\<U,V\>_{M,P},~~\forall~~ U,V\in H(\Omega)
\end{gathered}
\end{equation}
Now we will prove that $\Lambda_{A,\Sigma}'(U)V,~\Upsilon_{M,P}'(U)V$ are continuous functionals. \\
Let $U_l\to U$ in $H(\Omega)$, we show that $||\Lambda_{A,\Sigma}'(U_l)-\Lambda_{A,\Sigma}'(U)||_{\mathbb{L}(H(\Omega),\mathbb{R})}\to 0$ as $l\to\infty,$
and  $||\Upsilon_{M,P}'(U_l)-\Upsilon_{M,P}'(U)||_{\mathbb{L}(H(\Omega),\mathbb{R})}\to 0$ as $l\to\infty$ where 
$$\mathbb{L}(H(\Omega),\mathbb{R}):=\{\xi:H(\Omega)\to\mathbb{R}: \xi\in C(H(\Omega),\mathbf{R})\}$$
i.e.; ${\mathbb{L}(H(\Omega),\mathbb{R})}$  the set of all continuous functional from
$H(\Omega)$   to $\mathbb{R},$
since we know that 
$$||\Lambda_{A,\Sigma}'(U_l)-\Lambda_{A,\Sigma}'(U)||_{\mathbb{L}(H(\Omega),\mathbb{R})}=
\displaystyle{\sup_{||V||_{(A,\Sigma)}=1}|\Lambda_{A,\Sigma}'(U_l)V-\Lambda_{A,\Sigma}'(U)V|,~~\forall~~ U,V \in H(\Omega)}$$
$$|\Lambda_{A,\Sigma}'(U_l)V-\Lambda_{A,\Sigma}'(U)V|=$$ $2\left| 
\left(\int_{\Omega}\triangledown U_{l}.\triangledown V \,dx+\int_{\Omega}\<A(x)U_{l},V\> \,dx+\int_{\partial\Omega}\<\Sigma U_{l},V\>\,dx\right)
-\left(\int_{\Omega}\triangledown U.\triangledown V \,dx+\int_{\Omega}\<A(x)U,V\> \,dx+\int_{\partial\Omega}\<\Sigma U,V\>\,dx\right)
\right|
$
$$|\Lambda_{A,\Sigma}'(U_l)V-\Lambda_{A,\Sigma}'(U)V|\leq$$
$$ \int_{\Omega}|\triangledown U_{l}-\triangledown U|.|\triangledown V| \,dx
+\int_{\Omega}{\<A(x)(U_{l}-U),V\>}| \,dx+\int_{\partial\Omega}|\<\Sigma (U_{l}-U),V\>|\,dx$$
$$= \int_{\Omega}|\triangledown U_{l}-\triangledown U|.|\triangledown V| \,dx
+\int_{\Omega}|\sum_{i=1}^{k}(A(x)(U_{l}-U))_{j}v_i| \,dx+\int_{\partial\Omega}|\sum_{i=1}^{k}(\Sigma (U_{l}-U))_{j}v_{i}|\,dx$$
$$\leq \int_{\Omega}|\triangledown U_{l}-\triangledown U|.|\triangledown V| \,dx
+\int_{\Omega}\sum_{i=1}^{k}|(A(x)(U_{l}-U))_{j}v_i| \,dx+\int_{\partial\Omega}\sum_{i=1}^{k}|(\Sigma (U_{l}-U))_{j}v_{i}|\,dx $$
$$=\int_{\Omega}|\triangledown U_{l}-\triangledown U|.|\triangledown V| \,dx
+\sum_{i=1}^{k}\int_{\Omega}|(A(x)(U_{l}-U))_{j}v_i| \,dx+\sum_{i=1}^{k}\int_{\partial\Omega}|(\Sigma (U_{l}-U))_{j}v_{i}|\,dx $$
Where $$(A(x)(U_{l}-U))_{j}=\sum_{j=1}^{k}a_{ij}(x)(u_{lj}-u_{j})$$, and 
$$(\Sigma(x)(U_{l}-U))_{j}=\sum_{j=1}^{k}\sigma_{ij}(x)(u_{lj}-u_{j})$$

where 
$U=\{u_{1}\cdots,u_{k}\}$,  $U_{l}=\{u_{l1}\cdots,u_{lk}\},$ and $V=\{v_{1}\cdots,v_{k}\}$ 
so we have that 
$$\frac{1}{2}|\Lambda_{A,\Sigma}'(U_k)V-\Lambda_{A,\Sigma}'(U)V|\leq$$
$$\leq
\int_{\Omega}|\triangledown U_{l}-\triangledown U|.|\triangledown V| \,dx$$
$$+\sum_{j=1}^{k}\sum_{i=1}^{k}\int_{\Omega}|a_{ij}(x)(u_{lj}-u_{j})v_i| \,dx
+\sum_{j=1}^{k}\sum_{i=1}^{k}\int_{\partial\Omega}|\sigma_{ij}(x)(u_{lj}-u_{j})v_{i}|\,dx $$
$$\leq ^{ Holder~~ inequality} \left(\int_{\Omega}|\triangledown U_{l}-\triangledown U|^{2}\right)^{\frac{1}{2}}.
\left(\int_{\Omega}||\triangledown V|^{2}\,dx\right)^{\frac{1}{2}} $$
$$+\sum_{j=1}^{k}\sum_{i=1}^{k}\left(\int_{\Omega}|a_{ij}(x)||(u_{lj}-u_{j})|^{2}\,dx\right)^{\frac{1}{2}}
\left(\int_{\Omega}|a_{ij}(x)|v_i|^{2}\,dx\right)^{\frac{1}{2}}$$ $$
+\sum_{j=1}^{k}\sum_{i=1}^{k}\left(\int_{\partial\Omega}\sigma_{ij}(x)|u_{lj}-u_{j}|^{2}\,dx\right)^{\frac{1}{2}}
\left(\int_{\partial\Omega}\sigma_{ij}(x)|v_{i}|^{2}\,dx\right)^{\frac{1}{2}}
$$
So we have that 
$$|\Lambda_{A,\Sigma}'(U_l)V-\Lambda_{A,\Sigma}'(U)V|\leq2||U_{l}-U||_{(A,\Sigma)}||V||_{(A,\Sigma)}\to0,~~{\rm as~}l\to\infty$$
then, $||\Lambda_{A,\Sigma}'(U_l)-\Lambda_{A,\Sigma}'(U)||_{\mathbb{L}(H(\Omega),\mathbb{R})}\to 0$ as $l\to\infty$ so that 
$\Lambda_{A,\Sigma}'(U)$ is continuous functional.
similar argument we can prove that $\Upsilon_{M,P}'(U)V$ is  continuous functional. 
\end{proof}
\end{lemma}
   \begin{lemma}\label{ls1}
   Suppose $(A2,S2,C1-C3)$, and assumption\ref{ASMP} are holds then the functional
   $\Lambda_{A,\Sigma}$ is convex. 
      \end{lemma}
      \begin{proof}
       $$\Lambda_{A,\Sigma}(tU+(1-t)V)=||tU+(1-t)V)||_{(A,\Sigma)}^{2}\leq\left(||tU||_{(A,\Sigma)}+||(1-t)V)||_{(A,\Sigma)}\right)^{2}$$
       $$\leq t^{2}||U||_{(A,\Sigma)}^{2}+2t(1-t)||U||_{(A,\Sigma)}||V||_{(A,\Sigma)}+(1-t)^{2}||V||_{(A,\Sigma)}^{2}$$
       $$\leq t^{2}||U||_{(A,\Sigma)}^{2}+t(1-t)\left[||U||_{(A,\Sigma)}^{2}+||V||_{(A,\Sigma)}^{2}\right]+(1-t)^{2}||V||_{(A,\Sigma)}^{2}$$
       $$\leq t^{2}||U||_{(A,\Sigma)}^{2}+t||U||_{(A,\Sigma)}^{2}-t^{2}||U||_{(A,\Sigma)}^{2}+t||V||_{(A,\Sigma)}^{2})-t^{2}||V||_{(A,\Sigma)}^{2}
       +||V||_{(A,\Sigma)}^{2}-2t||V||_{(A,\Sigma)}^{2}+t^{2}||V||_{(A,\Sigma)}^{2}$$
       $$=t ||U||_{(A,\Sigma)}^{2}-t||V||_{(A,\Sigma)}^{2}+||V||_{(A,\Sigma)}^{2}=t||U||_{(A,\Sigma)}^{2}+(1-t)||V||_{(A,\Sigma)}^{2}
       =t\Lambda_{A,\Sigma}(U)+(1-t)\Lambda_{A,\Sigma}(V)$$
      \end{proof}

 \begin{lemma}\label{ls2}
  $\forall~U,V\in H(\Omega),$ then $\Lambda_{A,\Sigma}'(U)(V-U)\leq \Lambda_{A,\Sigma}(V)-\Lambda_{A,\Sigma}(U)$ 
 \end{lemma}
 \begin{proof}
  Since $\Lambda_{A,\Sigma}$ is convex from Lemma\ref{ls1}
  $$\Lambda_{A,\Sigma}(tV+(1-t)U)=\Lambda_{A,\Sigma}(U+t(V-U))\leq \Lambda_{A,\Sigma}(U)
  +t(\Lambda_{A,\Sigma}(V)-\Lambda_{A,\Sigma}(U)) \forall~U,V\in H(\Omega),~~\forall~t\in (0,1) $$
 $$\frac{ \Lambda_{A,\Sigma}(U+t(V-U))-\Lambda_{A,\Sigma}(U)}{t}\leq 
  +\Lambda_{A,\Sigma}(V)-\Lambda_{A,\Sigma}(U) \forall~U,V\in H(\Omega),~~\forall~t\in (0,1) $$
  $$\lim_{t\to 0}\frac{ \Lambda_{A,\Sigma}(U+t(V-U))-\Lambda_{A,\Sigma}(U)}{t}\leq 
  +\Lambda_{A,\Sigma}(V)-\Lambda_{A,\Sigma}(U) \forall~U,V\in H(\Omega),~~\forall~t\in (0,1) $$
  so we have that 
$$\Lambda_{A,\Sigma}'(U)(V-U)\leq \Lambda_{A,\Sigma}(V)-\Lambda_{A,\Sigma}(U)~~\forall~U,V\in H(\Omega).$$
  
 \end{proof}
\begin{theorem}\label{ths1}
Let $\Lambda_{A,\Sigma}$ be $G-$differentiable and convex, then $\Lambda_{A,\Sigma}$ is weakly lower-semi-continuous 
\end{theorem}
\begin{proof}
 Let $U_{l}\rightharpoonup U$ in $H(\Omega)$ since $\Lambda_{A,\Sigma}'$ continuous then 
 $\displaystyle\lim_{l\to\infty}\Lambda_{A,\Sigma}(U)'(U_{l})= \Lambda_{A,\Sigma}'(U)(U)$
 by \ref{ls2} we have that 
 $$\Lambda_{A,\Sigma}'(U)(U_{l}-U)\leq \Lambda_{A,\Sigma}(U_{l})-\Lambda_{A,\Sigma}(U)$$
 so we have that 
 $$\liminf_{l\to\infty}\Lambda_{A,\Sigma}'(U)(U_{l}-U)\leq \liminf_{l\to\infty}(\Lambda_{A,\Sigma}(U_{l})-\Lambda_{A,\Sigma}(U))$$
since the limit of the left hand side exist and equal zero then we get that 
$$0\leq\liminf_{l\to\infty}(\Lambda_{A,\Sigma}(U_{l})-\Lambda_{A,\Sigma}(U))$$
so we have that 
$$\Lambda_{A,\Sigma}(U)\leq\liminf_{l\to\infty}\Lambda_{A,\Sigma}(U_{l})$$
Therefore $\Lambda_{A,\Sigma}$ is weakly lower-semi-continuous
 \end{proof}
 \begin{lemma}\label{ls4}
  Suppose that (A1,S1,C1-C3) holds, then 
  \item[(i)] There exists $\delta>0$ such that   
  $$||U||_{(M,P)}\leq \sqrt{\delta^{-1}}||U||_{(A,\Sigma)}$$
  \item[(ii)] The norms $||.||_{H}(\Omega)$ and $||.||_{(A,\Sigma)}$ are equivalents in $H(\Omega)$
 \end{lemma}
 \begin{proof}
  Define $$\mathbb{S}:=\{U\in H(\Omega): ||U||_{(M,P)}=1\}$$ and 
  $$\delta:=\inf_{U\in\mathbb{S}}||U||_{(A,\Sigma)}^{2}=\inf_{U\in\mathbb{S}}\Lambda_{A,\Sigma}(U)^{2}$$
  \begin{claim}\label{C1}
   There exists $\^{U}\in\mathbb{S}$ such that 
   $$\Lambda_{A,\Sigma}(\^{U})=||\^{U}||_{(A,\Sigma)}=\sqrt{\delta}$$
   By the definition of $\delta$, there exists sequence $\{U_{n}\}_{n=1}^{\infty} \in \mathbb{S}$ such that 
   $$\Lambda_{A,\Sigma}(U_n)\to\sqrt{\delta}~~{\rm and}~~\Lambda_{A,\Sigma}(U_n)<\sqrt{\delta}+1$$
   Since $\{U_{n}\}_{n=1}^{\infty} \in \mathbb{S}$ the sequence $\{U_{n}\}_{n=1}^{\infty}$ is bounded in $(H(\Omega)),||.||_{H}$ 
   which implies that exist a subsequence $\{U_{n_{k}}\}_{k=1}^{\infty}$ of $\{U_{n}\}_{n=1}^{\infty}$ and $\^{U}$ such that 
   $U_{n_k}\rightharpoonup\^{U}$ weakly in $(H(\Omega)),||.||_{H}$. Consequently, from the compact embedding of $H^{1}(\Omega)$  into 
   $L^{2}(\Omega)$, we that $U_{n_k}\to\^{U}$ in $(L^{2}_{N}(\Omega)),||.||_{L^{2}_{N}}$ where 
   $||U||_{L^{2}_{N}}:=\sum_{i=1}^{N}||u_i||_{L^{2}(\Omega)}$. 
      Thus $\^{U}\in\mathbb{S}$. Furthermore, it follows from Lemma(\ref{l1}) that
      $$||\^{U}||_{(A,\Sigma)}\leq\liminf_{k\to\infty}||U_{n_k}||_{(A,\Sigma)}=\sqrt{\delta}$$
      Therefore 
       $$||\^{U}||_{(A,\Sigma)}\leq\liminf_{k\to\infty}||U_{n_k}||_{(A,\Sigma)}
       \leq\limsup_{k\to\infty}||U_{n_k}||_{(A,\Sigma)}$$
       since we know that 
  $$\sqrt{\delta}=\liminf_{k\to\infty}||U_{n_k}||_{(A,\Sigma)}=\limsup_{k\to\infty}||U_{n_k}||_{(A,\Sigma)}=
  \lim_{k\to\infty}||U_{n_k}||_{(A,\Sigma)}=||\^{U}||_{(A,\Sigma)}$$
    Thus 
  $$\Lambda_{A,\Sigma}(\^{U})=\sqrt{\delta}$$
    \end{claim}
\begin{claim}\label{C2}
$$\delta >0$$
 Clarey, $\delta\geq$. If $\delta=0$, then $||\^{U}||_{(A,\Sigma)}=0$ and thus $\^{U}=0$, which contradicts the assumption that 
 $\^{U}\in\mathbb{S}$. Therefore $\delta>0$
\end{claim}
Since every $W\in H(\Omega)$ can be written as $W=\frac{U}{||U||_{(M,P)}}$ for all $U\in H(\Omega)$
so that $||W||_{(A,\Sigma)}=\frac{||U||_{(A,\Sigma)}}{||U||_{(M,P)}}$
Since $W\in H(\Omega)$, so that $||W||_{(A,\Sigma)}=\Lambda_{A,\Sigma}(W)\geq\sqrt{\delta}$. 
so we have that  
$$\sqrt{\delta}\leq\frac{||U||_{(A,\Sigma}}{||U||_{(M,P)}}$$
It follows from claim \ref{C1} and claim\ref{C2}, then 
$$||U||_{(M,P)}\leq \sqrt{\delta^{-1}}||U||_{(A,\Sigma)} ~~\forall ~U\in H(\Omega)$$
The proof of  the norms $||.||_{H}(\Omega)$ and $||.||_{(A,\Sigma)}$ are equivalents in $H(\Omega)$,
by continuity of $\Lambda_{A,\Sigma}$ there exists a constant $\xi >0$ such that 
$$||U||_{(A,\Sigma)}\leq\sqrt{\xi}||U||_{H(\Omega)}~\forall ~~U\in H(\Omega).$$
We have that 
$$||U||_{H(\Omega)}^{2}\leq ||U||_{(A,\Sigma)}^{2} +||U||_{(M,P)}^{2},$$ we get that 
$$||U||_{H(\Omega)}\leq\left( ||U||_{(A,\Sigma)}^{2} +||U||_{(M,P)}^{2}\right)^{\frac{1}{2}}$$ 
Therefore by the first part we have that 
$$||U||_{H(\Omega)}\leq\left( (1+{\delta^{-1}})||U||_{(A,\Sigma)}^{2}\right)^{\frac{1}{2}},$$  we get that 

$$||U||_{H(\Omega)}\leq\left( 1+{\delta^{-1}}\right)^{\frac{1}{2}}||U||_{(A,\Sigma)}$$
Therefore 
$$||U||_{H(\Omega)}\leq\sqrt{ 1+{\delta^{-1}}}||U||_{(A,\Sigma)}~~\forall~~ U\in H(\Omega)$$
So we get that 
$$\sqrt{\xi^{-1}}||U||_{(A,\Sigma)}\leq||U||_{H(\Omega)}\leq\sqrt{ 1+{\delta^{-1}}}||U||_{(A,\Sigma)}~~\forall~~ U\in H(\Omega)$$
The equivalences desired 
 \end{proof}
\section{Main theorem}

 \begin{theorem}\label{ths}
Assume that $A2,S2,C1-C3$ as above. Then we have the following. 
\begin{description}
 \item[i] The eigensystem \eqref{Es3} has a sequence of real eigenvalues 
$$0<\mu_{1}\leq\mu_2\leq\mu_3\leq\cdots\leq\mu_j\leq\cdots\to\infty~as~j\to\infty$$ 
each eigenvalue has a finite-dimensional eigenspace. 
\item[ii] The eigenfunctions $\varphi_j$ corresponding to the eigenvalues $\mu_j$ from an $(A,\Sigma)-$orthogonal and $(M,P)-$orthonormal family in 
$[\mathbb{H}_{M,P}(\Omega)]^{\bot}$ ( a closed subspace of $H(\Omega)$
\item[iii] The normalized eigenfunctions provide a complete $(A,\Sigma)-$orthonormal basis of $[\mathbb{H}_{M,P}(\Omega)]^{\bot}$. 
Moreover, each function in $U\in [\mathbb{H}_{M,P}(\Omega)]^{\bot}$
has a unique representation of the from 
\begin{equation}\label{4}
\begin{gathered}
U=\displaystyle\sum^{\infty}_{j=1}c_j\varphi_j~with~c_j:=\frac{1}{\mu_j}\<U,\varphi_j\>_{(A,\Sigma)}=<U,\varphi_j>_{(M.P)} \\
||U||^{2}_{(A,\Sigma)}=\displaystyle\sum^{\infty}_{j=1}\mu_j|c_j|^2
\end{gathered}
\end{equation}
In addition, $$||U||^{2}_{(M,P)}=\displaystyle\sum^{\infty}_{j=1}|c_j|^2$$
\end{description}
\end{theorem}
\begin{proof}
 We will prove the existence of a sequence of real eigenvalues $\mu_{j}$ and the eigenfunctions $\varphi_{j}$ corresponding to 
 the eigenvalues that from an orthogonal family in $[\mathbb{H}_{M,P}(\Omega)]^{\bot}$ \\
 We show that $\Lambda_{A,\Sigma}$ attains its minimum on the constraint set 
 $$W_{0}=\{U\in [\mathbb{H}_{M,P}(\Omega)]^{\bot}:\Upsilon_{M,P}(U)=0\}$$
 Let $\alpha:=\displaystyle\inf_{U\in W_{0}},$ by using the continuity of the trace operator, 
 the Sobolev embedding theorem and the lower-semi-continuity of 
 $\Lambda_{A,\Sigma}.$\\
 Let $\{U_{l}\}_{l=1}^{\infty}$ be a minimizing sequence in $W_{0}$ for $\Lambda_{A,\Sigma},$  
 since $\displaystyle\lim_{l\to\infty}\Lambda_{A,\Sigma}(U_{l})=\alpha,$ we have that $\Lambda_{A,\Sigma}(U_{l})=||U_l||_{(A,\Sigma)},$
 by the definition of $\alpha$ we have that 
for all sufficiently large $l$, and for all $\epsilon >0$, then $||U_l||^{2}_{A,\Sigma}\leq \alpha+\epsilon$ by using the equivalent norm 
(See lemma\ref{ls4}) we have that 
there is exist $\beta$ such that 
$$||U_l||^{2}_{H(\Omega)}\leq\beta ||U_l||^{2}_{A,\Sigma}$$ so we have that 

$$||U_l||^{2}_{H(\Omega)}\leq\beta ||U_l||^{2}_{A,\Sigma}\leq\beta(\alpha+\epsilon),$$
so this sequence is bounded in $H^(\Omega)$. Thus it has a weakly convergent subsequence $\{U_{l_j}:j\geq 1\}$ which convergent 
weakly to limit $\hat{U}$ in $H(\Omega)$. From Rellich-Kondrachov theorem this subsequence convergent strongly to $\hat{U}$ in $L^{2}_{k}(\Omega),$
so $\hat{U}$ in $W_0$. Thus $\Lambda_{A,\Sigma}(\hat{U})=\alpha$ as the functional is weakly l.s.c. (see Theorem\ref{ths1}) \\

 Then there exists $\varphi_1$ such that $\Lambda_{A,\Sigma}(\varphi_1)=\alpha$.  Hence, $\Lambda_{A,\Sigma}$ attains
 its minimum at $\varphi_1$ and $\varphi_1$ satisfies the following 
\begin{equation}\label{ES04}
\begin{gathered}
 \int_{\Omega}\triangledown.\varphi_1\triangledown V\,dx+\int_{\Omega}\<A(x)\varphi_{1},V\>\,dx+\int_{\partial\Omega}\<\Sigma(x)\varphi_{1},V\>
 =\mu_{1}\left(\int_{\Omega}\<M(x)\varphi_{1},V\>\,dx+\int_{\partial\Omega}\<P(x)\varphi_{1},V\>\,dx\right)
\end{gathered}
\end{equation}
For all $V\in[\mathbb{H}_{(M,P)}(\Omega)]^{\bot} $ We see that $(\mu_{1},\varphi_{1})$ satisfies equation (\ref{Es4})
in weak sense and $\varphi_{1} \in W_{0}$ this implies that 
$\varphi_{1}\in [\mathbb{H}_{(M,P)}(\Omega)]^{\bot}$ by the definition of $W_{0}$, Now take $V=\varphi_{1}$ in equation (\ref{ES04}),
we obtain that the eigenvalue $\mu_{1}$ is the infimum $\alpha=\Lambda_{A,\Sigma}(\varphi_{1})=\mu_{1}$. This means that we could define $\mu_{1}$
by Rayleigh quotient $$\mu_{1}=\inf_{\substack{U\in H(\Omega)\\u\neq 0}}\frac{\Lambda_{A,\Sigma}(U)}{||U||^{2}_{(M,P)}}$$
Clearly, $\mu_{1}=\Lambda_{A,\Sigma}(\varphi_{1})\geq 0$. Indeed  assume that $\Lambda_{A,\Sigma}(\varphi_{1})=0$ then $|\triangledown \varphi_{1}|=0$ on $\Omega$
, hence $\varphi_{1}$ must be a constant that contradicts the assumptions imposed on $A(x)$. Thus $\mu_{1} >0$. \\

Now we show the existence of higher eigenvalues. \\
Define 
$$\mathbb{F}_{1}:W_{0}\to\mathbb{R}~{\rm by}~\mathbb{F}_{1}(U)=\<U,\varphi_{1}\>_{M,P}$$ 
we know that the kernel of $\mathbb{F}_{1}$ 
$$ker\mathbb{F}_{1}=\{U\in W_{0}:\mathbb{F}_{1}(U)=0\}=:W_{1}.$$
Since $W_{1}$ is the null-space of the continuous functional $\<.,\varphi_{1}\>_{M,P}$ on $[\mathbb{H}_{(M.P)}(\Omega)]^{\bot},$ 
$W_{1}$ is a closed subspace of $[\mathbb{H}_{(M,P)}(\Omega)]^{\bot}$, and it is therefore a Hilbert space itself under the same inner product 
$<.,.>_{(M,P)}$. Now we define 
$$\mu_{2}=\inf\{\Lambda_{A,\Sigma}(U):U\in W_{1}\}=\inf_{\substack{U\in W_{1}\\U\neq 0}}\frac{\Lambda_{A,\Sigma}(U)}{||U||^{2}_{(M,P)}}$$
Since $W_{1}\subset W_0$ then we have that $\mu_{1}\leq\mu_{2}$. Now we define 
$$\mathbb{F}_{2}:W_{1}\to\mathbb{R}~{\rm by}~\mathbb{F}_{2}(U)=\<U,\varphi_{2}\>_{M,P}$$
we know that the kernel of $\mathbb{F}_{2}$ 
$$ker\mathbb{F}_{2}=\{U\in W_{1}:\mathbb{F}_{2}(U)=0\}=:W_{2}.$$
Since $W_{2}$ is the null-space of the continuous functional $\<.,\varphi_{2}\>_{M,P}$ on $[\mathbb{H}_{(M.P)}(\Omega)]^{\bot},$ 
$W_{2}$ is a closed subspace of $[\mathbb{H}_{(M,P)}(\Omega)]^{\bot}$, and it is therefore a Hilbert space itself under the same inner product 
$<.,.>_{(M,P)}$. Now we define 
$$\mu_{3}=\inf\{\Lambda_{A,\Sigma}(U):U\in W_{2}\}=\inf_{\substack{U\in W_{2}\\U\neq 0}}\frac{\Lambda_{A,\Sigma}(U)}{||U||^{2}_{(M,P)}}$$
Since $W_{2}\subset W_{1}$ then we have that $\mu_{2}\leq\mu_{3}$.

Moreover, we can repeat the above arguments to show that 
$\mu_{3}$ is achieved at some $\varphi_{3}\in [\mathbb{H}_{(M,P)}(\Omega)]^{\bot}.$ \\
We let 
$$W_{3}=\{u\in W_{2}:<U,\varphi_{3}>_{(M,P)}=0\},$$
and 
$$\mu_{4}=\inf\{\Lambda_{A,\Sigma}(U):U\in W_{3}\}=\inf_{\substack{U\in W_{3}\\u\neq 0}}\frac{\Lambda_{A,\Sigma}(U)}{||U||^{2}_{(M,P)}}$$
Since $W_{3}\subset W_{2}$ then we have that $\mu_{3}\leq\mu_{4}$. Moreover, we can repeat the above arguments to show that 
$\mu_{4}$ is achieved at some $\varphi_{4}\in [\mathbb{H}_{(M,P)}(\Omega)]^{\bot}.$ \\
Proceeding inductively, ( In general we can define 
$$\mathbb{F}_{j}:W_{j-1}\to\mathbb{R}~{\rm by}~\mathbb{F}_{j}(U)=\<U,\varphi_{j}\>_{M,P}$$
we know that the kernel of $\mathbb{F}_{2}$ 
$$ker\mathbb{F}_{j}=\{U\in W_{j-1}:\mathbb{F}_{j}(U)=0\}=:W_{j}.$$
Since $W_{j}$ is the null-space of the continuous functional $\<.,\varphi_{j}\>_{M,P}$ on $[\mathbb{H}_{(M.P)}(\Omega)]^{\bot},$ 
$W_{j}$ is a closed subspace of $[\mathbb{H}_{(M,P)}(\Omega)]^{\bot}$, and it is therefore a Hilbert space itself under the same inner product 
$<.,.>_{(M,P)}$. 
Now we define 
$$\mu_{j+1}=\inf\{\Lambda_{A,\Sigma}(U):U\in W_{j}\}=\inf_{\substack{U\in W_{j}\\U\neq 0}}\frac{\Lambda_{A,\Sigma}(U)}{||U||^{2}_{(M,P)}}$$
In this way, we generate a sequence of eigenvalues 
$$0<\mu_{1}\leq\mu_{2}\leq\mu_{3}\leq\ldots\leq\mu_{j}\leq\ldots$$
whose associated $\varphi_{j}$ are $c-$orthogonal and $(M,P)-$orthonormal in $[H^{1}_{0}(\Omega)]^{\bot}$ \\ 
\textbf{Claim 1} $\mu_{j}\to\infty$ as $j\to\infty$
\begin{proof}
 By way of contradiction, assume that the sequence is bounded above by constant. Therefore, the corresponding sequence of eigenfunctions $\varphi_{j}$
is bounded in $H(\Omega)$ (i.e.; by the definition of the limit at $\infty$ $\forall~ S>0,~\exists N>0$ such that $|\varphi_{j}|>S,$ whenever $j>N$, 
the   ignition of the statement $\exists S>0$ such that $|\varphi_{j}|\leq S~\forall j$). 
By Rellich-Kondrachov theorem and the compactness of the trace operator, there is a Cauchy subsequence (which we again denote by $\varphi_{j}$
such that 
\begin{equation}\label{Es9}
\begin{gathered}
||\varphi_{j}-\varphi_{k}||^{2}_{(M,P)}\to 0.
\end{gathered}
\end{equation}

Since the $\varphi_{j}$ are $(M,P)-$orthonormal, we have that 
$||\varphi_{j}-\varphi_{k}||^{2}_{(M,P)}=||\varphi_{j}||^{2}_{(M,P)}+||\varphi_{k}||^{2}_{(M.P)}=2>0$, 
if$j\neq k,$ which contradicts 
equation (\ref{Es9}). Thus, $\mu_{j}\to\infty.$ we have that each $\mu_{j}$ occurs only finitely many times. 
\end{proof}
\textbf{Claim 2} Each eigenvalue $\mu_{j}$ has a finite-dimensional eigenspace. \\ 
\begin{proof}
By way of contradiction, assume that each eigenvalue $\mu_{j}$ has infinite-dimensional eigenspace. 
let $\mu$ has corresponding sequence of eigenfunctions $\{\varphi_{1},\varphi_{2},...,\varphi_{j},...\}$ 
we know that $\mu=||\varphi_{1}||^{2}_{(A,\Sigma)}=...=||\varphi_{j}||^{2}_{(A,\Sigma)}=...$, this contradicts  claim 1
therefore, each eigenvalue has a finite-dimensional eigenspace

\end{proof}
We will show that the normalized eigenfunctions provide a complete orthonormal basis of $[H^{1}_{0}(\Omega)]^{\bot}$. 
Let $$\psi_{j}=\frac{1}{\sqrt{\mu_{j}}}\varphi_{j},$$
so that $||\psi_{j}||^{2}_{(A,\Sigma)}=1$ \\
\textbf{Claim 3} 
The sequence $\{\psi_{j}\}_{j\geq1}$ is a maximal $(A,\Sigma)-$orthonormal family of $[\mathbb{H}_{(M.P)}(\Omega)]^{\bot}$.
(we know that the set maximal $(A,\Sigma)-$orthonormal if and only if it is  complete orthonormal basis)
\begin{proof}
 By way of contradiction, assume that the sequence $\{\psi_{j}\}_{j\geq1}$ is not maximal, then there exists a $\xi\in 
 [\mathbb{H}_{(M.P)}(\Omega)]^{\bot},$ and 
$\xi\not\in \{\psi_{j}\}_{j\geq1},$ 
such that $||\xi||^{2}_{(A,\Sigma)}=1$ and $<\xi,\psi_{j}>_{c}=0~\forall~ j$, i.e.;
 $$0=\<\xi,\psi_{j}\>_{(A,\Sigma)}=\<\xi,\frac{1}{\sqrt{\mu_{j}}}\varphi_{j}\>_{(A,\Sigma)}=$$
 $$\frac{1}{\sqrt{\mu_{j}}}\<\xi,\varphi_{j}\>_{(A,\Sigma)}=^{(~by~\ref{ES04})}\frac{\mu_{j}}{\sqrt{\mu_{j}}}\<\xi,\varphi_{j}\>_{(M,P)}=
\mu_{j}\<\xi,\frac{1}{\sqrt{\mu_{j}}}\varphi_{j}\>_{(M,P)}={\mu_{j}}<\xi,\psi_{j}>_{\partial},$$  since $\mu_{j}>0~\forall~ j$. Therefore 
$\<\xi,\psi_{j}\>_{(M.P)}=0$. We have that $\xi\in W_{j}~\forall~ j\geq 1$. It follows from the definition of $\mu_{j}$ that 
$$\mu_{j}\leq\frac{||\xi||^{2}_{(A,\Sigma)}}{||\xi||^{2}_{(M,P)}}=\frac{1}{||\xi||^{2}_{(M,P)}}~\forall ~j\geq 1.$$ 
Since we know from claim 1 that $\mu_{j}\to\infty$ as $j\to\infty$  we have that $||\xi||^{2}_{(M,P)}=0.$ Therefore $\xi=0$ a.e in $\Omega$,
which contradicts the definition of $\xi.$ Thus the sequence $\{\psi_{j}\}_{j\geq1}$ 
is a maximal $(A,\Sigma)-$orthonormal family of $[H_{(M,P)}(\Omega)]^{\bot},$
so the sequence $\{\psi_{j}\}_{j\geq1}$  provides a complete orthonormal basis of $[H_{(M,P)}(\Omega)]^{\bot};$
that is, for any $U\in[\mathbb{H}_{(A,\Sigma)}(\Omega)]^{\bot}$,\\ \\
$U=\displaystyle\sum^{\infty}_{j=1}d_{j}\psi_{j}$  with $d_{j}=\<U,\psi_{j}\>_{(A,\Sigma)}=\frac{1}{\sqrt{\mu_{j}}}\<U,\varphi_{j}\>_{(A,\Sigma)},$
and $||U||^{2}_{(A,\Sigma)}=\displaystyle\sum^{\infty}_{j=1}|d_{j}|^2$\\ \\

$$U=\displaystyle\sum^{\infty}_{j=1}d_{j}\frac{1}{\sqrt{\mu_{j}}}\varphi_{j},$$ now let 
$$c_{j}=d_{j}\frac{1}{\sqrt{\mu_{j}}}=\frac{1}{\mu_{j}}\<U,\varphi_{j}\>_{(A,\Sigma)}=^{(\ref{ES04})}\<U,\varphi_{j}\>_{(M,P)}.$$
Therefore, 
$$U=\displaystyle\sum^{\infty}_{j=1}c_{j}\varphi_{j},$$ and 
$$||U||^{2}_{(A,\Sigma)}=\displaystyle\sum^{\infty}_{j=1}|c_{j}|^{2}||\varphi_{j}||^{2}_{(A,\Sigma)}=\displaystyle\sum^{\infty}_{j=1}\mu_{j}|c_{j}|^{2}$$
\end{proof}
\textbf{Claim 4} 
We shall show that 
$$||U||^{2}_{(M.P)}=\displaystyle\sum^{\infty}_{j=1}|c_{j}|^{2}$$
\begin{proof}
$$||U||^{2}_{(M,P)}=<u,u>_{(M,P)}=
\<\displaystyle\sum^{\infty}_{j=1}c_{j}\varphi_{j},\displaystyle\sum^{\infty}_{k=1}c_{k}\varphi_{k}\>_{(M,P)}
=\displaystyle\sum^{\infty}_{j=1}c_{j}\displaystyle\sum^{\infty}_{k=1}c_{k}\<\varphi_{j},\varphi_{k}\>_{(M,P)}
=\displaystyle\sum^{\infty}_{j=1}|c_{j}|^{2}.$$
Thus $$||U||^{2}_{(M,P)}=\displaystyle\sum^{\infty}_{j=1}|c_{j}|^{2}$$
\end{proof}
The following result gives a variational characterization of the eigenvalues and a splitting of the space 
$[\mathbb{H}_{(M,P)}(\Omega)]^{\bot},$ (and, hence, of $H(\Omega)$ which will be needed in the proofs of the result on nonlinear problems.\\ \\
\textbf{Corollary 1}
 Assume that $A1,A2, S1, S2, M1, P1$ and assumption\ref{ASMP}, holds. Then we have the following.
\begin{description}
 \item[i] For all $U\in H(\Omega) ,$
\begin{equation}\label{Es10}
\begin{gathered}
\mu_{1}\left(\int_{\Omega}\<M(x)U,U\>\,dx+\int_{\partial\Omega}\<P(x)U,U\>\,dx\right)
\leq\int_{\Omega}|\triangledown U|^2\,dx+\int_{\Omega}\<A(x)U,U\>\,dx+\int_{\partial\Omega}\<\Sigma(x)U,U\>\,dx,
\end{gathered}
\end{equation}
where $\mu_{1} >0$ is the least Steklov-Robin eigenvalue for system (\ref{Es3}). If equality holds in (\ref{Es10}), then $U$ is a multiple of 
an eigenfunction of equation (\ref{Es3}) corresponding to $\mu_{1}$
 \item[ii] For every $V\in\oplus_{i\leq j}E(\mu_{i}),$ and $W\in\oplus_{i\geq j+1}E(\mu_{i}),$ we have that 
\begin{equation}\label{Es11}
\begin{gathered}
||V||^{2}_{(A,\Sigma)}\leq\mu_{j}||V||^{2}_{(M.P)} ~~{\rm and}~~||W||^{2}_{(A,\Sigma)}\geq\mu_{j+1}||W||^{2}_{(M,P)} 
\end{gathered}
\end{equation}
where $E(\mu_{i})$ is the $\mu_{i}$-eigenspace and $\oplus_{i\leq j}E(\mu_{i})$ is span of the eigenfunctions associated 
to eigenvalues up to $\mu_{j}$
\end{description}
\begin{proof}
 If $U=0$, then the inequality (\ref{Es10}) holds. otherwise, if  $0\neq U\in H(\Omega),$
then $U=U_{1}+U_{2},$ where $U_{1}\in[H_{(M,P)}(\Omega)]^{\bot} $, and $U_{2}\in H_{(M,P)}(\Omega).$
Therefore, by the $(A,\Sigma)-$orthogonality, and the characterization of
 $\mu_{1}$ (i.e.; $\mu_{1}||U_{1}||^{2}_{(M,P)}\leq||u_{1}||^{2}_{(A,\Sigma)}$) we get that 
$$\mu_{1}||U||^{2}_{(M,P)}=\mu_{1}\left(||U_{1}||^{2}_{(M,P)}+\cancel{||U_{2}||^{2}_{(M,P)}}^{0}\right)
\leq ||U_{1}||^{2}_{(A,\Sigma)}\leq ||U||^{2}_{(A,\Sigma)}$$.
Thus, the inequality (\ref{Es10}) holds. \\
Now assume we have that 
$$||U||^{2}_{(A,\Sigma)}=\mu_{1}||U||^{2}_{(M,P)}\implies\mu_{1}= \frac{||U||^{2}_{(A,\Sigma)}}{||u||^{2}_{(M,P)}}$$
we know that $\mu_{1}=\frac{||\varphi_{1}||^{2}_{(A,\Sigma)}}{||\varphi_{1}||^{2}_{(M,P)}},$
where $\varphi_{1}$ the eigenfunction corresponding to $\mu_{1}$,
Therefore, $U$ is a multiple of an eigenfunction of system (\ref{Es3}) corresponding to $\mu_{1}.$\\
The inequalities (\ref{Es10}) by Theorem\ref{ths} we have that 
$$||V||^{2}_{(A,\Sigma)}=\displaystyle\sum^{\infty}_{i=1}\mu_i|c_i|^2~~\forall~~ V\in\oplus_{i\leq j}E(\mu_{i})$$.
Now let  $\mu_{j}=\max\mu~\forall i\leq j,$ then we have that 
$$||V||^{2}_{(A,\Sigma)}=\displaystyle\sum^{\infty}_{i=1}\mu_i|c_i|^2\leq\max\mu\displaystyle\sum^{\infty}_{i=1}|c_i|^2
=\mu_{j}||v||^{2}_{(M,P)}~\forall~V\in\oplus_{i\leq j}E(\mu_{i})$$

$$||W||^{(A,\Sigma)}_{c}=\displaystyle\sum^{\infty}_{j=1}\mu_j|c_j|^2~~\forall W\in\oplus_{i\leq j}E(\mu_{i})$$.
Now let  $\mu_{j+1}=\min\mu~\forall~ i\geq j+1,$ then we have that 
$$||W||_{(A,\Sigma)}=\displaystyle\sum^{\infty}_{j=1}\mu_j|c_j|^2\geq\min\mu\displaystyle\sum^{\infty}_{j=1}|c_j|^2=
\mu_{j+1}||w||^{2}_{(M,P)}~\forall~W\in\oplus_{i\geq j+1}E(\mu_{i})$$

\end{proof}
The following proposition shows the principality of the first eigenvalue $\mu_{1}.$
\begin{proposition}\label{Ep1}
The first eigenvalue $\mu_{1}$ is simple if and only if the associated eigenfunction $\varphi_{1}$ does not changes sign (i.e.; 
$\varphi_{1}$ is strictly positive or strictly negative in $\Omega$ .
\end{proposition}
\begin{proof}
 Assume that the first eigenvalue $\mu_{1}$ is simple, we will show that associated eigenfunction $\varphi_{1}$ does not changes sign in $\Omega$, 
 suppose it does and let $\varphi_{1}=\varphi^{+}_{1}+\varphi^{-}_{1},$ where $\varphi^{+}_{1}=\max\{\varphi_{1},0\},$
 and $\varphi^{-}_{1}=\min\{0,\varphi_{1}\}$\\ 
If $\varphi_{1}\in H(\Omega)$. Then $\varphi^{+}_{1},\varphi^{-}_{1}\in H^{1}(\Omega)$
proof of that we know that $\varphi^{+}_{1}=\frac{1}{2}(\varphi_{1}+|\varphi_{1}|)$
clearly $\varphi^{+}_{1}\in L^{2}_{k}(\Omega)$, define 
$$V_{\epsilon}=(\varphi_{1}^{2}+\epsilon^{2})^{\frac{1}{2}}-\epsilon$$ 
$$|\varphi_{1}|=\displaystyle\lim_{\epsilon\to 0}V_{\epsilon},$$ we will show that 
$$D^{i}V_{\epsilon}=\frac{\varphi_{1}}{(\varphi_{1}^{2}+\epsilon^{2})^{\frac{1}{2}}}D_{i}\varphi_{1}\longrightarrow^{L^{2}_{k}(\Omega)}~sign D^{i}\varphi_{1}$$
$\forall\epsilon >0$, then $0 \leq V_{\epsilon}\leq |\varphi_{1}|$
since
 $$V_{\epsilon}^{2}=\left((\varphi_{1}^{2}+\epsilon^{2})^{\frac{1}{2}}-\epsilon\right)^{2}=\varphi_{1}^{2}+\epsilon^{2}-2(\varphi_{1}^{2}+\epsilon^{2})^{\frac{1}{2}}\epsilon+\epsilon^{2}
=\varphi_{1}^{2}+2\epsilon(\epsilon-(\varphi_{1}^{2}+\epsilon^{2})^{\frac{1}{2}})\leq\varphi_{1}^{2},$$
therefore  $V_{\epsilon}\leq |\varphi_{1}|$, 
$$\displaystyle\lim_{\epsilon\to 0}||\frac{\varphi_{1}}{(\varphi_{1}^{2}+\epsilon^{2})^{\frac{1}{2}}}D_{i}\varphi_{1}-sign D^{i}\varphi_{1}||_{L^{2}_{k}(\Omega)}=0$$
Therefore, 
$$\frac{\varphi_{1}}{(\varphi_{1}^{2}+\epsilon^{2})^{\frac{1}{2}}}D_{i}\varphi_{1}\longrightarrow^{L^{2}_{k}(\Omega)}~sign D^{i}\varphi_{1}$$

Thus, $\varphi_{1}^{+}\in H^{1}(\Omega)$, similar $\varphi_{1}^{-}\in H^{1}(\Omega)$\\
By the characterization of $\mu_{1}$ it follows that  
$$\<\varphi_{1},\varphi_{1}\>_{(A,\Sigma)}=\mu_{1}\<\varphi_{1},\varphi_{1}\>_{(M,P)},$$ since  $\varphi_{1}^{+}\in H(\Omega)$, and 
$\varphi_{1}^{-}\in H(\Omega),$ we have that 
$$\mu_{1}\<\varphi_{1}^{+},\varphi_{1}^{+}\>_{(M,P)}\leq\<\varphi_{1}^{+},\varphi_{1}^{+}\>_{(A,\Sigma)},$$
$$\mu_{1}\<\varphi_{1}^{-},\varphi_{1}^{-}\>_{(M,P)}\leq\<\varphi_{1}^{-},\varphi_{1}^{-}\>_{(A,\Sigma)}.$$ Therefore
$$0\leq~ \<\varphi_{1}^{+},\varphi_{1}^{+}\>_{(A,\Sigma)}+\<\varphi_{1}^{-},\varphi_{1}^{-}\>_{(A,\Sigma)}-
\mu_{1}\<\varphi_{1}^{+},\varphi_{1}^{+}\>_{(M,P)}
-\mu_{1}\<\varphi_{1}^{-},\varphi_{1}^{-}\>_{(M,P)}=$$$$
\<\varphi_{1}^{+}+\varphi_{1}^{-},\varphi_{1}^{+}+\varphi_{1}^{-}\>_{(A,\Sigma)}
- \mu_{1}\<\varphi_{1}^{+}+\varphi_{1}^{-},\varphi_{1}^{+}+\varphi_{1}^{-}\>_{(M,P)}
=\<\varphi_{1},\varphi_{1}\>_{(A,\Sigma)}-\mu_{1}\<\varphi_{1},\varphi_{1}\>_{(M,P)}=0.$$  
It follows that $\varphi_{1}^{+},$ and $\varphi_{1}^{-}$ are also eigenfunctions corresponding to $\mu_{1}$ we have that
  $\varphi_{1}^{+}>0~a.e~in~\Omega,$ and $\varphi_{1}^{-}<0~a.e~in~\Omega,$ which is impossible since $\mu_{1}$ it is simple. 
Thus $\varphi_{1}$ does not change sign in $\Omega.$\\
Assume $\varphi_{1}$ change sign, then $\varphi_{1}^{+},$ and $\varphi_{1}^{-}$ are also eigenfunctions corresponding to $\mu_{1}$ and 
they are linearly independent. Hence, $\mu_{1}$  is not simple.
On the other hand, suppose that $\mu_{1}$ is not simple, and let $\varphi$ and $\psi$ be two eigenfunctions corresponding 
to $\mu_{1}$ they are linearly independent. If $\varphi$ or  $\psi$ changes sign, then the proposition is proved. Otherwise, 
supposing without loss of generality that $\varphi$ and   $\psi$ positive, we will prove that there exists $a\in\mathbb{R}$ such that 
the eigenfunction (corresponding to $\mu_{1}$) $\varphi+a\psi$ changes sign. Indeed, suppose that, for all $\alpha\in\mathbb{R},$ 
$\varphi+\alpha\psi$ does not change.\\
Let the function $h:\mathbb{R}\to\mathbb{R}$ be define by 
$$h(\alpha)=\int\varphi+\alpha\int\psi.$$
Since $h$ is continuous, there exists $a\in\mathbb{R}$ such that 
$$h(a)=\int\varphi+a\int\psi=0.$$
Hence,  which contradicts the fact  $\varphi$ and $\psi,$ are linearly independent. 
Thus, $\varphi+a\psi,$ changes sign. The proof is complete. 
\end{proof}
\begin{remark}
Note that if we have smooth data and $\partial\Omega$ in proposition \ref{Ep1}, then the eigenfunction $\varphi_{1}(x)$ on $\partial\Omega$ as well, 
by the boundary point lemma (see for example \cite{Eva99}). or Hopf's Boundary Point Lemma.
\end{remark}

\end{proof}

\end{document}